\documentclass[11pt, letterpaper,reqno]{amsart}

\usepackage{graphicx}
\usepackage[linktocpage=true,colorlinks,citecolor=blue,linkcolor=blue,urlcolor=black]{hyperref}
\usepackage{fullpage}
\usepackage{amssymb}
\usepackage{cite}
\usepackage{amsmath}
\usepackage{latexsym}
\usepackage{amscd}
\usepackage{tikz}
\usepackage{amsthm}
\usepackage{mathrsfs}
\usepackage{url}
\usepackage[utf8]{inputenc}
\usepackage[english]{babel}
\usepackage{amsfonts}
\usepackage{todonotes}
\usepackage{mathtools}
\usepackage{verbatim}
\usepackage[justification=centering]{caption}

\numberwithin{equation}{section}
\theoremstyle{plain}
\newtheorem{theorem}{Theorem}

\newtheorem{lemma}[theorem]{Lemma}

\newtheorem*{theorem*}{Theorem}

\newtheorem{corollary}[theorem]{Corollary}

\theoremstyle{remark}

\theoremstyle{definition}

\theoremstyle{remark}

\numberwithin{equation}{section}

\begin{document}
\title{On connected components with many edges}
\author{Sammy Luo}

\thanks{Department of Mathematics, Stanford University, Stanford, CA 94305, USA. \\
\indent Supported by NSF GRFP Grant DGE-1656518. Email: {\tt sammyluo@stanford.edu}.}
	\maketitle

\begin{abstract}
We prove that if $H$ is a subgraph of a complete multipartite graph $G$, then $H$ contains a connected component $H'$ satisfying $|E(H')||E(G)|\geq |E(H)|^2$. We use this to prove that every three-coloring of the edges of a complete graph contains a monochromatic connected subgraph with at least $1/6$ of the edges. We further show that such a coloring has a monochromatic circuit with a fraction $1/6-o(1)$ of the edges. This verifies a conjecture of Conlon and Tyomkyn. Moreover, for general $k$, we show that every $k$-coloring of the edges of $K_n$ contains a monochromatic connected subgraph with at least $\frac{1}{k^2-k+\frac{5}{4}}\binom{n}{2}$ edges.
\end{abstract}

\section{Introduction}

A classical observation of Erd\H{o}s and Rado is that in any two-coloring of the edges of the complete graph $K_n$, one of the color classes forms a connected graph. In \cite{gyarfas1977partition}, Gyárfás proves the following generalization of this observation: For any $k\geq 2$, in every $k$-coloring of the edges of $K_n$, there is a monochromatic connected component with at least $n/(k-1)$ vertices. This observation has since been extended in numerous ways, such as by replacing $K_n$ with a graph of high minimum degree \cite{gyarfas2017mindeg} or with a nearly-complete bipartite graph \cite{deblasio2020bipartite}, or adding a constraint on the diameter of the large monochromatic component \cite{Ruszinko2012diameter}. 
See \cite{Gyarfas2011Survey} for a survey of earlier work on the subject.

The arguments used in this subject tend to focus on sparse spanning structures like double stars. As such, there is a surprising lack of progress on the corresponding question about edges in a monochromatic connected component. That is, what is the largest value of $M=M(n,k)$ 
such that every $k$-edge-coloring of $K_n$ has a monochromatic connected component with at least $M$ edges?

This question was raised by Conlon and Tyomkyn in \cite{conlon2021ramsey}, in the context of determining the multicolor Ramsey numbers for trails (see Section~\ref{subsec:trails} for the relevant definitions). 
After showing that $M(n,2)=\frac{2}{9}n^2+o(n^2)$, they sketch a simple argument that shows $M(n,k)\geq \frac{1}{16k^2}n^2+O(n)$ for all $k$; with a slightly more careful analysis, their argument in fact yields $M(n,k)\geq \frac{1}{4k^2}n^2+O(n)$. In the other direction, they examine a construction of Gyárfás in \cite{Gyarfas2011Survey} to show that $M(n,k)\leq \frac{1}{2k(k-1)}n^2+O(n)$ for infinitely many values of $k$ (specifically, when $k-1$ is a prime power), conjecturing that this upper bound is tight in the case $k=3$.

In this note, we improve the general lower bound on $M(n,k)$, as well as a corresponding lower bound for the trail Ramsey problem, bringing it to asymptotically within a factor $1-O\left(\frac{1}{k^2}\right)$ of the upper bound for infinitely many values of $k$.
\begin{theorem}
\label{thm:colorlb}
For any $k\geq 2$, in every $k$-coloring of the edges of $K_n$, there is a monochromatic connected component with at least $M(n,k)\geq \frac{1}{k^2-k+\frac{5}{4}}\binom{n}{2}$ edges.
\end{theorem}
By building on the overarching ideas in \cite{Gyarfas2011Survey} and introducing some key new insights, we manage to strengthen this lower bound in the case $k=3$ to prove the tight lower bound conjectured by Conlon and Tyomkyn.

\begin{theorem}
\label{thm:color3}
In every $3$-coloring of the edges of $K_n$, there is a monochromatic component containing at least a sixth of the edges. That is, $M(n,3)\geq \lceil\frac{1}{6}\binom{n}{2}\rceil$. Moreover, for $n$ sufficiently large, this bound is sharp.
\end{theorem}
While equality holds in this bound for sufficiently large $n$, there are, as we will see from the proof, small values of $n$ for which $M(n,3)>\lceil\frac{1}{6}\binom{n}{2}\rceil$. In particular, equality holds for all $n\geq 18$, but not for $n=17$.

The key result we use is a new inequality that may be of independent interest. Given a subgraph $H$ of a complete multipartite graph $G$, it relates the edge counts of $G$ and $H$ to the largest edge count of a connected component $H'$ of $H$.
\begin{theorem}
\label{thm:main}
Let $G$ be a complete $r$-partite graph for some $r\geq 2$, and let $H$ be a subgraph of $G$. Then $H$ contains a connected component $H'$ satisfying
\[
|E(H')|\geq \frac{|E(H)|^2}{|E(G)|}.
\]
\end{theorem}
In effect, this result settles the density analogue of the coloring question of determining $M(n,k)$. Instead of partitioning the edges of a graph $G$ into color classes, we are fixing a subgraph $H$ of $G$ with a given fraction $\delta=\frac{|E(H)|}{|E(G)|}$ of its edges, and asking about the component of $H$ with the most edges. Theorem~\ref{thm:main} can then be restated as: if $|E(H)| = \delta |E(G)|$, then $H$ contains a connected component $H'$ with $|E(H')|\geq \delta^2 |E(G)|$. Equality can be attained asymptotically when $\delta=\frac{1}{k}$ for any positive integer $k\geq 2$: Let $V(G)=V_1\cup\cdots\cup V_r$ be an $r$-partition of $V(G)$, split each $V_i$ into $k$ (roughly) equally sized vertex sets $\{V_{i,j}\}_{j=1}^k$, and for $1\leq j\leq k$, let $H_j$ be the subgraph of $G$ induced on $\bigcup_{i=1}^r V_{i,j}$. Then indeed, $H=\bigcup_{j=1}^k H_j$ is a subgraph of $G$ whose components have edge counts
\[
|E(H_j)|=\frac{1}{k^2}|E(G)|+O(n) = \frac{|E(H)|^2}{|E(G)|} + O(n).
\]

The simplest case of Theorem~\ref{thm:main}, when $G=K_n$, already immediately implies an improvement over previously known lower bounds on $M(n,k)$.

\begin{corollary}\label{cor:weaklb}
In every $k$-coloring of the edges of $K_n$, there is a monochromatic connected component with at least $\frac{1}{k^2}\binom{n}{2}$ edges.
\end{corollary}

\begin{proof}
In a $k$-coloring of the edges of $K_n$, one of the color classes has at least $\frac{1}{k}\binom{n}{2}$ edges. Taking $H$ to be this color class in Theorem~\ref{thm:main} with $G=K_n$ yields Corollary~\ref{cor:weaklb}.
\end{proof}

The proof of Theorem~\ref{thm:color3} requires a more detailed argument, but similarly follows from this one case of Theorem~\ref{thm:main}, while the proof of Theorem~\ref{thm:colorlb} requires the use of Theorem~\ref{thm:main} in full generality.

In the next section, we give an elementary proof of Theorem~\ref{thm:main}. We then study the coloring version of the problem, as well as the corresponding trail Ramsey problem, in Section~\ref{sec:color}.

\section{\texorpdfstring{Proof of Theorem~\ref{thm:main}}{Proof of Theorem 3}}
\label{sec:mainpf}

In the case $G=K_n$, the proof of Theorem~\ref{thm:main} is very simple, but nevertheless includes a key insight: Instead of taking a component that maximizes the number of edges right away, we consider a component $H'$ with maximum \textbf{average degree} $\bar{d}(H')=\frac{|E(H')|}{\frac{1}{2}|V(H')|}$. Two observations are crucial here: First, by the so-called \emph{generalized mediant inequality}, this highest average degree must be at least the average degree of the whole graph $H$, since 
\[
\frac{|E(H)|}{\frac{1}{2}|V(H)|}=\frac{\sum|E(H_i)|}{\frac{1}{2}\sum|V(H_i)|},
\]
where the sums are over connected components $H_i$ of $H$, and the right hand side is a generalized mediant of the average degrees of the individual components. 
Second, the number of vertices of $H'$ is at least one more than its maximum degree, so $|V(H')|\geq \bar{d}(H')+1$. Letting $\delta = \frac{|E(H)|}{|E(G)|}$, we then obtain the bound
\begin{align*}
    |E(H')|&\geq \frac{1}{2}(\bar{d}(H'))(\bar{d}(H')+1) = \binom{\bar{d}(H')+1}{2}\\
    &\geq \binom{\bar{d}(H)+1}{2} = \binom{\frac{2|E(H)|}{n}+1}{2}  \\
    &= \binom{\delta(n-1)+1}{2} \geq \delta^2 \binom{n}{2} \\
    &= \delta^2 |E(G)| = \frac{|E(H)|^2}{|E(G)|},
\end{align*}
as desired.

The general case will use both of these observations in a modified setting. Instead of the average degree, we will work with a slightly different quantity whose denominator is a weighted vertex count; we will then, perhaps counterintuitively, lower bound this weighted vertex count by the modified analogue of the average degree, in order to obtain the bound we seek. The core of our proof is the following general inequality.
\begin{lemma}
\label{lem:weightcs}
If $a_1,\dots,a_r$ and $b_1,\dots,b_r$ are nonnegative real numbers, then
\[
\left(\left(\sum_{i=1}^r a_i\right)\left(\sum_{i=1}^r b_i\right) - \sum_{i=1}^r a_i b_i \right)^2 \geq \left(\left(\sum_{i=1}^r a_i\right)^2 - \sum_{i=1}^r a_i^2 \right)\left(\left(\sum_{i=1}^r b_i\right)^2 - \sum_{i=1}^r b_i^2 \right).
\]
\end{lemma}
\begin{proof}
The Cauchy-Schwarz inequality yields
\begin{align*}
    \left(\sum_{i=1}^r a_i\right)\left(\sum_{i=1}^r b_i\right) - \sum_{i=1}^r a_i b_i 
    \geq \left(\sum_{i=1}^r a_i\right)\left(\sum_{i=1}^r b_i\right) - \sqrt{\left(\sum_{i=1}^r a_i^2\right)\left(\sum_{i=1}^r b_i^2\right)}.
\end{align*}
This last quantity is clearly nonnegative, since $\left(\sum_{i=1}^r a_i\right)^2\geq \sum_{i=1}^r a_i^2$ and $\left(\sum_{i=1}^r b_i\right)^2\geq \sum_{i=1}^r b_i^2$. So,
\begin{align*}
    &\left(\left(\sum_{i=1}^r a_i\right)\left(\sum_{i=1}^r b_i\right) - \sum_{i=1}^r a_i b_i \right)^2 \geq \left(\left(\sum_{i=1}^r a_i\right)\left(\sum_{i=1}^r b_i\right) - \sqrt{\left(\sum_{i=1}^r a_i^2\right)\left(\sum_{i=1}^r b_i^2\right)}\right)^2 \\
    &= \left(\sum_{i=1}^r a_i\right)^2\left(\sum_{i=1}^r b_i\right)^2 + \left(\sum_{i=1}^r a_i^2\right)\left(\sum_{i=1}^r b_i^2\right) - 2 \left(\sum_{i=1}^r a_i\right)\left(\sum_{i=1}^r b_i\right)\sqrt{\left(\sum_{i=1}^r a_i^2\right)\left(\sum_{i=1}^r b_i^2\right)} \\
    &\geq \left(\sum_{i=1}^r a_i\right)^2\left(\sum_{i=1}^r b_i\right)^2 + \left(\sum_{i=1}^r a_i^2\right)\left(\sum_{i=1}^r b_i^2\right) - \left(\sum_{i=1}^r a_i\right)^2\left(\sum_{i=1}^r b_i^2\right) - \left(\sum_{i=1}^r b_i\right)^2\left(\sum_{i=1}^r a_i^2\right) \\
    &= \left(\left(\sum_{i=1}^r a_i\right)^2 - \sum_{i=1}^r a_i^2 \right)\left(\left(\sum_{i=1}^r b_i\right)^2 - \sum_{i=1}^r b_i^2 \right),
\end{align*}
where the last inequality is an application of the AM-GM inequality.
\end{proof}

Given a graph $G$ and vertex sets $S,T\subseteq V(G)$, let
\[
e_G(S,T)=\#\{(s,t)\in S\times T:\: s\sim_G t\}.
\]
We write $e(S,T)$ for $e_G(S,T)$ when the graph in question is unambiguous. Let $G[S]$ denote the induced subgraph of $G$ on the vertex set $S$. Lemma~\ref{lem:weightcs} immediately implies the following.
\begin{corollary}
\label{cor:weightcs}
Let $G$ be a complete multipartite graph. For any $S, T\subseteq V(G)$, we have
\[
e(S,T)^2\geq 4|E(G[S])||E(G[T])|.
\]
\end{corollary}
\begin{proof}
Suppose that $G$ is $r$-partite. Let $V(G)=V_1\cup \cdots \cup V_r$ be a partition of the vertices of $G$ into $r$ independent sets. Let $a_i=|V_i\cap S|$ and $b_i=|V_i\cap T|$, so
\[
e(S,T)=\sum_{\substack{1\leq i,j\leq r\\ i\neq j}} a_{i}b_{j}=\left(\sum_{i=1}^r a_i\right)\left(\sum_{i=1}^r b_i\right) - \sum_{i=1}^r a_i b_i,
\]
while
\[
|E(G[S])|=\frac{1}{2}\left(\left(\sum_{i=1}^r a_i\right)^2 - \sum_{i=1}^r a_i^2 \right), \qquad |E(G[T])|=\frac{1}{2}\left(\left(\sum_{i=1}^r b_i\right)^2 - \sum_{i=1}^r b_i^2 \right).
\]
Then Lemma~\ref{lem:weightcs} indeed yields $e(S,T)^2\geq 4|E(G[S])||E(G[T])|$, as desired.
\end{proof}

\begin{proof}[Proof of Theorem~\ref{thm:main}]
Let $V=V(G)=V_1\cup\cdots\cup V_r$ be a partition of the vertices of $G$ into $r$ independent sets, let $H_1,\dots, H_k$ be the connected components of $H$, and let $V_{i,\ell}=V_i\cap H_\ell$. We have
\[
|E(H)|=\sum_{j=1}^k |E(H_\ell)|, \qquad |V_i|=\sum_{\ell=1}^k |V_{i,\ell}|\: \text{ for all } i.
\]
For any subset $S\subseteq V$, define $f(S)=\frac{1}{2}e_G(S,V)$. Note that
\[
f(S)=\frac{1}{2}\sum_{\substack{1\leq i,j\leq r\\i\neq j}} |V_{i}\cap S||V_{j}|,
\]
so $f(S)$ can be viewed as a weighted vertex count for $S$, with the property that
\[
\sum_{\ell=1}^k f(V(H_\ell))=\frac{1}{2}\sum_{\ell=1}^k e_G(V(H_\ell),V)=\frac{1}{2}e_G(V,V)=|E(G)|.
\]
Then by the generalized mediant inequality, for some $H'=H_\ell$ we have
\[
\frac{|E(H')|}{f(V(H'))}\geq \frac{\sum_{\ell=1}^k |E(H_\ell)|}{\sum_{\ell=1}^k f(V(H_\ell))} = \frac{|E(H)|}{|E(G)|}.
\]
Now, Corollary~\ref{cor:weightcs} applied with $S=V(H')$, $T=V$ yields $f(V(H'))^2\geq |E(G[V(H')])||E(G)|\geq |E(H')||E(G)|$, so that
\[
\frac{|E(H)|}{|E(G)|}\leq \frac{|E(H')|}{f(V(H'))}\leq \sqrt{\frac{|E(H')|}{|E(G)|}},
\]
which rearranges to the desired inequality.
\end{proof}

\section{Coloring problems}
\label{sec:color}
We can now apply Theorem~\ref{thm:main} to the corresponding coloring problems. We have already seen that applying Theorem~\ref{thm:main} with $G=K_n$ readily yields the simple lower bound on $M(n,k)$ given in Corollary~\ref{cor:weaklb}. We now discuss how to improve this lower bound, before turning to a closely related problem on monochromatic trails and circuits.

\subsection{Lower bound on $M(n,k)$ for general $k$}
Our strategy for improving the lower bound on $M(n,k)$ is as follows: First, assuming that no monochromatic component has too many edges, we show that in the color with the highest density (say, red), we can upper bound the number of vertices covered by any set of components (or else we can finish with an average degree argument on the rest of the red components). Through a smoothing argument, this yields an upper bound on the sum of squares of the number of vertices in each red component, and thus a lower bound on the number of edges in the complete multipartite graph $G$ formed by deleting from $K_n$ all edges within the vertex sets of the red components. We finish by applying Theorem~\ref{thm:main} to $G$ to find a non-red component with many edges.

\begin{proof}[Proof of Theorem~\ref{thm:colorlb}]
Fix a $k$-coloring $\chi$ of the edges of $K_n$, and suppose that the largest number of edges in a monochromatic connected component in this coloring is $z\binom{n}{2}$. Without loss of generality, let red be the color with the most edges, so there are at least $\frac{1}{k}\binom{n}{2}$ red edges in our coloring. Let $\mathcal{C}_1=\{R_1,R_2,\dots,R_m\}$ be the set of red components, with
\begin{equation}
    \label{eqn:totalvtx}
    |V(R_1)|\geq |V(R_2)|\geq \cdots \geq |V(R_m)|, \qquad \sum_{i=1}^m |V(R_i)|=n.
\end{equation}
Let $x=\frac{1}{\binom{n}{2}}\sum_{i=1}^m |E(R_i)|$, so $x\geq \frac{1}{k}$. By assumption, $|E(R_i)|\leq z\binom{n}{2}$ for $1\leq i\leq m$. As in the proof of Corollary~\ref{cor:weaklb}, applying Theorem~\ref{thm:main} with $K_n$ as $G$ and its red color class (i.e. the spanning subgraph formed by its red edges) as $H$ yields a red component with at least $x^2 \binom{n}{2}$ edges, so by assumption we have $z\geq x^2\geq \frac{1}{k^2}$. Define $\delta=k-\frac{1}{\sqrt{z}}\geq 0$, so $z=\frac{1}{(k-\delta)^2}$.

For any $j\in [1,m-1]$, let $G_j$ be the complete graph on $\bigcup_{i=j+1}^m V(R_i)$, with its coloring induced by $\chi$. Applying Theorem~\ref{thm:main} with $G_j$ as $G$ and the red color class of $G_j$ as $H$ yields a red connected component $H'=R_i$ with $|E(H')|\geq \frac{|E(H)|^2}{|E(G_j)|}$. Since
\[
|E(H)|=\sum_{i=j+1}^m |E(R_i)|=x\binom{n}{2}-\sum_{i=1}^{j}|E(R_i)| \geq x\binom{n}{2}-jz\binom{n}{2},
\]
while $|E(G_j)|=\binom{|V(G_j)|}{2}=\binom{n-\sum_{i=1}^{j}|V(R_i)|}{2}$, we have 
\[
z\binom{n}{2}\geq |E(H')|\geq \frac{\max(x-jz,0)^2 \binom{n}{2}^2}{\binom{n-\sum_{i=1}^{j}|V(R_i)|}{2}}.
\]
Since $\frac{\max(x-jz,0)}{\sqrt{z}}\leq \frac{x}{\sqrt{z}}\leq 1$, this implies
\[
\binom{n-\sum_{i=1}^{j}|V(R_i)|}{2}\geq \frac{\max(x-jz,0)^2}{z}\binom{n}{2}\geq \binom{\frac{\max(x-jz,0)}{\sqrt{z}}n}{2}.
\]
Since $n-\sum_{i=1}^{j}|V(R_i)|\geq 1$, and the function $f(X)=\binom{X}{2}$ is increasing for $X\geq 1$, this then implies
\begin{equation}
    \label{eqn:vtxbounds}
    \sum_{i=1}^{j}|V(R_i)|\leq (1-x/\sqrt{z} + j\sqrt{z})n \qquad \text{for all }j\in [1,m-1].
\end{equation}
We can now give an upper bound on $\sum_{i=1}^m \binom{|V(R_i)|}{2}$ by solving the corresponding convex optimization problem. The proof of the following technical lemma will be deferred until the end of the section.

\begin{lemma}
\label{lem:smoothing}
Let $x,z>0$. Subject to the constraints $v_1\geq \cdots \geq v_m \geq 0$, $\sum_{i=1}^m v_i= 1$, and
\[
\sum_{i=1}^j v_i \leq 1-x/\sqrt{z}+j\sqrt{z} \qquad \text{for all }j\in [1,m-1],
\]
the quantity $\sum_{i=1}^m v_i^2$ is maximized when $v_1=1-x/\sqrt{z}+\sqrt{z}$, $v_i=\sqrt{z}$ for $2\leq i\leq \lfloor \frac{x}{z}\rfloor$, $v_{\lfloor \frac{x}{z}\rfloor+1}=x/\sqrt{z}-\lfloor \frac{x}{z}\rfloor \sqrt{z}$, and $v_i=0$ for $i>\lfloor \frac{x}{z}\rfloor+1$.
\end{lemma}

Let $v_i=\frac{|V(R_i)|}{n}$. Since \eqref{eqn:totalvtx} and \eqref{eqn:vtxbounds} hold, we can apply Lemma~\ref{lem:smoothing} to obtain
\[
\sum_{i=1}^m v_i^2\leq (1-x/\sqrt{z}+\sqrt{z})^2 + \left(\left\lfloor \frac{x}{z}\right\rfloor-1\right) z +\left(\left(x/z-\left\lfloor \frac{x}{z}\right\rfloor\right) \sqrt{z}\right)^2 \leq (1-x/\sqrt{z}+\sqrt{z})^2 + (x/z-1)z,
\]
so that we have
\begin{align*}
    \sum_{i=1}^m \binom{|V(R_i)|}{2} &=  \frac{n^2 \sum_{i=1}^m v_i^2 -n }{2}  \leq \frac{n^2 ((1-x/\sqrt{z}+\sqrt{z})^2 + (x/z-1)z) - n}{2}.
\end{align*}
Finally, let $G$ be the complete $m$-partite graph obtained from $K_n$ by removing all edges within each $V(R_i)$, with its coloring induced by $\chi$. There are no red edges in $G$, so by the pigeonhole principle, one of the $k-1$ remaining colors has at least $\frac{1}{k-1}|E(G)|$ edges in $G$. Let $H$ be the spanning subgraph of $G$ induced by the edges in that color. Applying Theorem~\ref{thm:main} then yields a monochromatic connected component $H'$ with at least $\frac{1}{(k-1)^2}|E(G)|$ edges. Then by assumption we have
\begin{align*}
    z &\geq \frac{1}{(k-1)^2}\frac{|E(G)|}{\binom{n}{2}} = \frac{1}{(k-1)^2} \left(1- \frac{\sum_{i=1}^m \binom{|V(R_i)|}{2}}{\binom{n}{2}} \right) \\
    &\geq \frac{1}{(k-1)^2} \left(1- \frac{n^2((1-x/\sqrt{z}+\sqrt{z})^2 + (x/z-1)z) - n}{n^2-n} \right)\\
    &\geq \frac{1}{(k-1)^2} \left(1- (1-x/\sqrt{z}+\sqrt{z})^2 - (x/z-1)z \right), 
\end{align*}
which rearranges to give
\[
1-(k-1)^2 z \leq (1-x/\sqrt{z}+\sqrt{z})^2 + (x/z-1)z=\frac{1}{z}x^2 - (1+2/\sqrt{z})x+(1+2\sqrt{z}).
\]
The right hand side is a quadratic in $x$ that is decreasing for $x\leq \frac{z}{2}+\sqrt{z}$. Since by assumption $\sqrt{z}\geq x\geq \frac{1}{k}$, we then have
\[
1-(k-1)^2 z \leq \frac{1}{z k^2} - (1+2/\sqrt{z})\frac{1}{k}+(1+2\sqrt{z}).
\]
Substituting in $z=\frac{1}{(k-\delta)^2}$ yields
\[
1-\frac{(k-1)^2}{(k-\delta)^2}\leq \frac{(k-\delta)^2}{k^2} - \frac{1+2(k-\delta)}{k}+1+\frac{2}{k-\delta},
\]
which upon rearrangement becomes
\begin{align*}
    0 &\leq (k-1)^2 k^2 + (k-\delta)^4 - k(k-\delta)^2 - 2(k-\delta)^3 k + 2(k-\delta) k^2 \\
    &= -k^3+k^2-k\delta^2+2k^3\delta -2k\delta^3+\delta^4\\
    &=(k-\delta^2)^2 - k(k^2-\delta^2)(1-2\delta).
\end{align*}
This implies
\[
1-2\delta \leq \frac{(k-\delta^2)^2}{k(k^2-\delta^2)}<\frac{1}{k},
\]
so that $\delta>\frac{k-1}{2k}$. Thus, the coloring $\chi$ contains a monochromatic connected component with at least $z\binom{n}{2}$ edges, where
\[
z=\frac{1}{(k-\delta)^2}>\frac{1}{(k-\frac{1}{2}+\frac{1}{2k})^2}\geq \frac{1}{k^2-k+\frac{1}{4}+1-\frac{1}{4k}+\frac{1}{4k^2}} \geq \frac{1}{k^2-k+\frac{5}{4}},
\]
as desired.
\end{proof}

\begin{proof}[Proof of Lemma~\ref{lem:smoothing}]
Since the feasible region is compact, there exists a choice of the $v_i$ such that the desired maximum is attained. Fix such a maximizing choice of the $v_i$.
For $j\in [m]$, let $A_j$ denote the given constraint on $\sum_{i=1}^j v_i$, and let $S$ be the set of $j$ for which equality holds in $A_j$. Let $m'=\lfloor \frac{x}{z}\rfloor$. Since $\sum_{i=1}^m v_i=1$, the constraints $A_j$ for $j>m'$ cannot be tight, so $S\subseteq [m']$. For any $i<j$ and any $\varepsilon\in (0,v_j]$, replacing $(v_i,v_j)$ with $(v_i+\varepsilon, v_j-\varepsilon)$ increases the value of $\sum_{i=1}^m v_i^2$, so by maximality, no such ``smoothing'' operation is possible. That is, we can assume the following \textbf{equality condition} for all $(i,j)$ with $i<j$: Either $v_i=v_{i-1}$, or $v_j=v_{j+1}$, or $S\cap [i,j-1]\neq \emptyset$. If $S=[m']$, then we are in exactly the maximizing case described (since by the equality conditions for $(m'+1,i)$ we have $v_i=0$ for all $i\geq m'+2$), so assume otherwise.

First, suppose there is some $i_0\in [m']$ such that $v_{i_0}<\sqrt{z}$, and pick the smallest such index $i_0$. Let $i_1\in [m]$ be the largest index such that $v_{i_1}>0$, and note that $i_1>i_0$ since $\sum_{i=1}^{i_0} v_i<1$. Then we have $v_{i_0}<v_{i_0-1}$ and $v_{i_1}>v_{i_1+1}$. But as $\sum_{i=1}^j v_i\leq (1-x/\sqrt{z}+(j-1)\sqrt{z}) + v_j < (1-x/\sqrt{z}+(j-1)\sqrt{z})+ \sqrt{z}$ for any $j\geq i_0$, we have $S\cap [i_0,i_1]=\emptyset$, contradicting the equality condition for $(i_0,i_1)$. So, we can assume $v_i\geq \sqrt{z}$ for all $i\leq m'$. In particular, if $j\in S$ for some $j\in [m']$, then we recursively obtain $v_i=\sqrt{z}$ for all $i\in [j+1,m']$, so $[j,m']\subseteq S$.

Thus, we can assume $1\notin S$, i.e. $v_1<1-x/\sqrt{z}+\sqrt{z}$. By the equality condition for $(1,2)$, we must then have $v_2=v_3$. Let $j_0\geq 3$ be the smallest index such that $v_{j_0+1}\neq v_{j_0}$. By the equality condition for $(1,j_0)$, there is some $j_1\in S\cap [2,j_0-1]$. Then
\[
1-x/\sqrt{z}+j_1 \sqrt{z} = \sum_{i=1}^{j_1} v_i = v_1+(j_1-1) v_2 < 1-x/\sqrt{z}+\sqrt{z}+(j_1-1)v_2,
\]
which implies $v_2>\sqrt{z}$. But then $\sum_{i=1}^{j_1+1} v_i=1-x/\sqrt{z}+\sqrt{z}+j_1 v_2 > 1-x/\sqrt{z}+(j_1+1) \sqrt{z}$, violating $A_{j_1+1}$. This is a contradiction, so in fact $S=[m']$, and we are in the desired maximizing case.
\end{proof}

We remark that in Gyárfás's construction, after removing all edges contained in the vertex set of each red component, each non-red component is left with approximately $\frac{1}{(k-1)^2} (1-\frac{1}{k-1}) \binom{n}{2} = \frac{k-2}{(k-1)^3} \binom{n}{2}$ edges. Since $\frac{k-2}{(k-1)^3}=\frac{1}{k^2-k+1+\frac{1}{k-2}}$, this suggests that the method used to prove Theorem~\ref{thm:colorlb} cannot show a lower bound on $M(n,k)$ better than $\frac{1}{k^2-k+1}\binom{n}{2}$ without introducing additional ideas.

\subsection{Multicolor Ramsey numbers of trails and circuits}
\label{subsec:trails}
A \emph{trail} is a walk without repeated edges, and a \emph{circuit} is a trail with the same first and last vertex. The ($k$-color) Ramsey problem for trails is the question of finding the largest $m$ such that every $k$-coloring of the edges of $K_n$ contains a monochromatic trail of length $m$.

Answering a question of Osumi \cite{osumi2021ramsey}, Conlon and Tyomkyn \cite{conlon2021ramsey} show that every $2$-coloring of the edges of $K_n$ contains a monochromatic circuit with at least $\frac{2}{9}n^2+O(n^{3/2})$ edges, and this is asymptotically tight. For the case of general $k$, they observe that by deleting a forest in each color class of a $k$-coloring of $K_n$ to make each color class Eulerian (i.e. ensuring every vertex has even degree in each color), one can reduce this Ramsey problem to a variant of the problem of determining $M(n,k)$. Where previously we colored the edges of $K_n$ and found a large monochromatic component, we now apply the same procedure to the graph obtained by deleting at most $kn$ edges from $K_n$. We now prove a lower bound for the general case of this problem, analogous to the bound on $M(n,k)$ given in Theorem~\ref{thm:colorlb}.

\begin{corollary}
\label{thm:trailsweak}
Every $k$-coloring of the edges of $K_n$ contains a monochromatic circuit (and hence a monochromatic trail) of length at least $\frac{1}{k^2-k+\frac{5}{4}}\binom{n}{2}+O_k(n)$.
\end{corollary}
\begin{proof}
Fix a $k$-coloring of $E(K_n)$. As in \cite{conlon2021ramsey}, we can remove from each color class a forest that meets all odd degree vertices, leaving a coloring where every color class, and hence every monochromatic connected component, is Eulerian. Let $\chi$ be the resulting partial $k$-coloring of $E(K_n)$, where the (at most $kn$) removed edges are left uncolored, and let $z\binom{n}{2}$ be the largest number of edges in a monochromatic component in this coloring. It suffices to show that $z\geq \frac{1}{k^2-k+\frac{5}{4}}+O_k(n^{-1})$, since every monochromatic component is Eulerian, and thus contains an Eulerian circuit. Fix a color (say, red) with $x\binom{n}{2}$ edges, where $x\geq \frac{1}{k} \frac{\binom{n}{2}-nk}{\binom{n}{2}}=\frac{1}{k}-\frac{2}{n-1}$. As before, applying Theorem~\ref{thm:main} with $G=K_n$ immediately yields $z\geq x^2\geq \frac{1}{k^2}-O(n^{-1})$. Note that this means there is a monochromatic circuit with at least $\left(\frac{1}{k^2}-O(n^{-1})\right)\binom{n}{2}=\frac{1}{k^2}\binom{n}{2}+O(n)$ edges.

To improve this lower bound further, we proceed as in the proof of Theorem~\ref{thm:colorlb}. Letting $R_1,\dots,R_m$ be the red components, such that \eqref{eqn:totalvtx} holds, we obtain \eqref{eqn:vtxbounds} in the same manner as before, so we can once again apply Lemma~\ref{lem:smoothing} to obtain the same upper bound on $\sum_{i=1}^m \binom{|V(R_i)|}{2}$ in terms of $x$ and $z$. Let $G$ be the complete $m$-partite graph with parts $V(R_1),\dots,V(R_m)$, with its partial coloring induced by $\chi$. Since $G$ has no red edges, and at most $kn$ edges are uncolored, one of its color classes $H$ has at least $\frac{1}{k-1}(|E(G)|-kn)=\left(\frac{1}{k-1}-O_k(n^{-1})\right)|E(G)|$ edges. Applying Theorem~\ref{thm:main} as before yields
\begin{align*}
    z &\geq \left(\frac{1}{k-1}-O_k(n^{-1})\right)^2\frac{|E(G)|}{\binom{n}{2}} 
    \geq \left(\frac{1}{(k-1)^2}-O_k(n^{-1})\right) \left(1- \frac{\sum_{i=1}^m \binom{|V(R_i)|}{2}}{\binom{n}{2}} \right) \\
    &\geq \left(\frac{1}{(k-1)^2+O_k(n^{-1})}\right) \left(1- (1-x/\sqrt{z}+\sqrt{z})^2 - (x/z-1)z \right).
\end{align*}
Letting $z=\frac{1}{(k-\delta)^2}$ and noting that $x\geq \frac{1}{k}-O(n^{-1})$, we can perform the same rearrangements and substitutions as in the proof of Theorem~\ref{thm:colorlb}, simply separating out the $O_k(n^{-1})$ terms at each step, to derive the inequality $1-2\delta < \frac{1}{k} + O_k(n^{-1})$, and thus 
\[
z\geq \frac{1}{k^2-k+\frac{5}{4}-O_k(n^{-1})}= \frac{1}{k^2-k+\frac{5}{4}}+O_k(n^{-1}).
\]
Then the coloring $\chi$ contains a monochromatic connected component, and hence a monochromatic Eulerian circuit, with at least $z\binom{n}{2}\geq \frac{1}{k^2-k+\frac{5}{4}}\binom{n}{2}+O_k(n)$ edges, as claimed.
\end{proof}

\subsection{Three colors}
In this section, we prove the lower and upper bounds on $M(n,3)$ in separate lemmas in order to establish Theorem~\ref{thm:color3}. We then discuss the behavior of $M(n,3)$ for small values of $n$, and conclude by describing how to adapt our proofs to obtain asymptotically tight bounds for the Ramsey numbers of trails and circuits in three colors.

\begin{lemma}
\label{lem:3lb}
Every $3$-coloring of the edges of $K_n$ contains a monochromatic component with at least $\lceil \frac{1}{6}\binom{n}{2}\rceil$ edges.
\end{lemma}

\begin{proof}
Let $G=K_n$, and call the three colors red, green, and blue. 
First, suppose one of the color classes (say, red) is connected. If there are at least $\frac{1}{6}|E(G)|$ red edges, we are done. Otherwise, the other two colors together have at least $\frac{5}{6}|E(G)|$ edges, so one of them has at least $ \frac{5}{12}|E(G)|$ edges. Applying Theorem~\ref{thm:main} with $G=K_n$ to the graph in that color then gives a monochromatic component with at least $(\frac{5}{12})^2 |E(G)| >\frac{1}{6}|E(G)|$ edges, so we are again done.

Thus, we can assume every color has at least two components. Without loss of generality, let red be the color with the most edges, so the red graph $H$ has at least $\frac{1}{3}|E(G)|$ edges. Let $H'$ be the component of $H$ with the highest average degree, so $|V(H')|\geq \bar{d}(H')+1\geq \frac{1}{3}n$. We can assume $|V(H')|\leq \frac{1}{2}n$; otherwise, we would have $|E(H')|\geq \frac{1}{2}|V(H')|\bar{d}(H') > \frac{1}{6}|E(G)|$. Let $V_1=V(H')$ and $V_2=V(G)\setminus V_1$, and let $G'$ be the bipartite graph induced by $G$ between $V_1$ and $V_2$, so $G'$ has at least $|V_1||V_2|\geq \frac{2}{9}n^2>\frac{1}{3}|E(G)|$ edges, all of which must be green or blue.

Fix an edge in $G'$ and consider the monochromatic component $C_1$ of $G'$ containing this edge. Without loss of generality, assume $C_1$ is green. Suppose $C_1$ covers all of $V_1$. Since every green edge in $G'$ intersects $V_1$, this means there is exactly one green component of $G'$ with a nonzero number of edges. Since there are at least two green components in $G$, there is a vertex $v\in V_2$ not in $C_1$. Then all edges between $v$ and $V_1$ must be blue, so all vertices of $V_1$ are in the same blue component in $G'$. This implies that there is also exactly one blue component of $G'$ with nonzeroly many edges. Thus, all edges of $G'$ are in one of two monochromatic components, one of which then has at least $\frac{1}{2}|E(G')|>\frac{1}{6}|E(G)|$ edges, as desired.

\begin{figure}[h]
    \centering
    \includegraphics[scale=0.8]{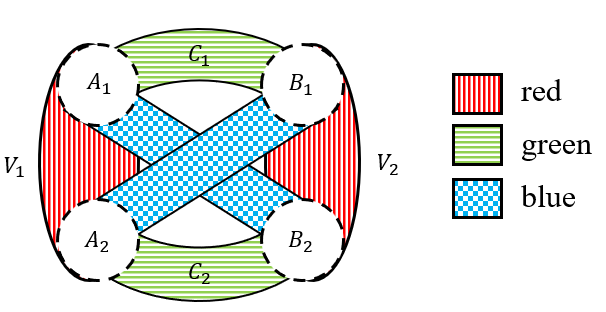}
    \caption{Lower bound: Exactly two components of each color}
    \label{fig:3color}
\end{figure}

We are likewise done if $C_1$ covers all of $V_2$, so we can assume $V_1\setminus C_1$ and $V_2\setminus C_1$ are both nonempty. Let $A_1=V_1\cap C_1$, $B_1=V_2\cap C_1$, $A_2=V_1\setminus A_1$, $B_2=V_2\setminus B_1$. Then all edges between $A_1$ and $B_2$, or between $A_2$ and $B_1$, can only be blue. Then there are at most two blue components, and thus by assumption exactly two. This in turn implies that all edges between $A_1$ and $B_1$, or between $A_2$ and $B_2$, can only be green, so there are exactly two green components $C_1$ and $C_2$. Finally, all edges between $B_1$ and $B_2$ can only be red, so we conclude that there are exactly two red components, $V_1$ and $V_2$; see Figure~\ref{fig:3color}.

Since each of the three colors has exactly two components, one of the components has at least $\frac{1}{6}|E(G)|$ edges, as desired.
\end{proof}

\begin{lemma}
\label{lem:3ub}
For sufficiently large $n$, there exists a $3$-coloring of the edges of $K_n$ such that every monochromatic component contains at most $\lceil \frac{1}{6} \binom{n}{2}\rceil$ edges.
\end{lemma}

\begin{proof}
\begin{figure}[h]
    \centering
    \includegraphics[scale=0.8]{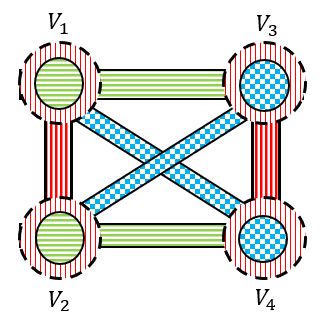}
    \caption{Initial construction for the $M(n,3)$ upper bound}
    \label{fig:3ub}
\end{figure}

We consider the following modification of a construction by Gyárfás: Let $V=V_1\cup V_2\cup V_3\cup V_4$ be a partition of the vertices of $G=K_n$ into four parts, with
\[
\left\lceil \frac{n}{4}\right \rceil = |V_1| \geq \cdots \geq |V_4| = \left\lfloor \frac{n}{4}\right \rfloor.
\]
Letting $E(U,W)$ denote the set of edges between vertex sets $U$ and $W$, color $E(V_1,V_2)\cup E(V_3,V_4)$ red, $E(V_1,V_3)\cup E(V_2,V_4)$ green, and $E(V_1,V_4)\cup E(V_2,V_3)$ blue. For large enough $n$, $e(V_i,V_j)\leq \lceil \frac{1}{6}\binom{n}{2}\rceil$ for $1\leq i<j\leq 4$, so it remains to extend this partial coloring by coloring the edges within each of the $V_i$ such that each monochromatic component has at most $\lceil \frac{1}{6}\binom{n}{2}\rceil$ edges at the end. It is natural to attempt the simple approach of distributing the edges within each $V_i$ as evenly as possible among the three colors. However, the possibility of a slight difference in size among the $V_i$ can yield a $\Theta(n)$ difference among the numbers of edges in each component if we are not sufficiently careful; a different coloring strategy will turn out to be simpler to analyze.

We call an extension of the above partial coloring \emph{nice} if each green or blue component contains exactly $\lceil\frac{1}{6}\binom{n}{2}\rceil$ edges. At least one nice coloring exists: we can color exactly $\lceil \frac{1}{6}\binom{n}{2}\rceil - e(V_1,V_3)$ of the edges within $V_1$ and $\lceil \frac{1}{6}\binom{n}{2}\rceil - e(V_2,V_4)$ of the edges within $V_2$ green, and exactly $\lceil \frac{1}{6}\binom{n}{2}\rceil - e(V_2,V_3)$ of the edges within $V_3$ and $\lceil \frac{1}{6}\binom{n}{2}\rceil - e(V_1,V_4)$ of the edges within $V_4$ blue (there are enough edges within each $V_i$ to do this when $n$ is sufficiently large), and color all remaining edges red. See Figure~\ref{fig:3ub} for a diagram of this coloring. Fix a nice coloring where the larger of the two red components contains as few edges as possible.

Suppose one of the red components in this coloring, without loss of generality the one on $V_1\cup V_2$, has more than $\lceil\frac{1}{6}\binom{n}{2}\rceil$ edges. Then $V_3\cup V_4$ must have less than $\lceil\frac{1}{6}\binom{n}{2}\rceil$ red edges. Without loss of generality let $V_1$ contain a red edge $e$. If either $V_3$ contains a green edge, or $V_4$ contains a blue edge, we can switch the color of that edge with edge $e$, preserving the sizes of the green and blue components while decreasing the size of the larger red component by one. Otherwise, $V_3$ is entirely blue and red, and $V_4$ is entirely green and red. Since the red component on $V_3\cup V_4$ has less than $\lceil\frac{1}{6}\binom{n}{2}\rceil$ edges, neither $V_3$ nor $V_4$ can be entirely red (for sufficiently large $n$, $\lfloor \frac{n}{4}\rfloor^2 + \binom{\lfloor \frac{n}{4}\rfloor}{2}>\lceil \frac{1}{6}\binom{n}{2}\rceil$). Then if $V_2$ contains a red edge, we can similarly switch two edges to reduce the size of the larger red component by one, so we can assume $V_2$ is entirely green and blue. But then there is one component of each color (including the larger red component) that does not have any edges within $V_2\cup V_3\cup V_4$, which means there are at least $3\lceil\frac{1}{6}\binom{n}{2}\rceil$ edges incident to $V_1$, a contradiction since
\[
e(V_1,V)\leq \binom{\lceil \frac{n}{4}\rceil}{2}+\left\lceil \frac{n}{4}\right\rceil\left(n-\left\lceil \frac{n}{4}\right\rceil\right) < 3\left\lceil\frac{1}{6}\binom{n}{2}\right\rceil,
\]
for sufficiently large $n$. Thus indeed there is a construction of this form where every monochromatic component has at most $\lceil\frac{1}{6} \binom{n}{2}\rceil$ edges.
\end{proof}
Combining Lemmas~\ref{lem:3lb} and~\ref{lem:3ub} yields Theorem~\ref{thm:color3} as desired.

The construction in the proof of Lemma~\ref{lem:3ub} is well-defined for all $n\geq 46$. A more careful analysis, splitting into cases based on the value of $n$ modulo $4$ and then using a more explicit construction in each case, shows that in fact $M(n,3)=\left\lceil\frac{1}{6}\binom{n}{2}\right\rceil$ for all $n\geq 11$ except $n=13,17$. However, the lower bound from Lemma~\ref{lem:3lb} is not sharp for some of these small values of $n$. Closely inspecting each step of our proof for $n=17$, for example, we can deduce that $M(17,3)=24$, instead of the expected $23$; indeed, all of the bounds before the final step in the proof of Lemma~\ref{lem:3lb} are loose enough to yield a component with at least $24$ edges unless we are in the case depicted in Figure~\ref{fig:3color}, where one of the vertex sets $A_i$ or $B_i$ has size $5$, and the other three sets have size $4$. The set of size $5$ then contains $10$ internal edges, some $4$ of which are in the same color, yielding a component with $4+20=24$ edges as claimed. This shows a genuine difference in the behavior of $M(n,3)$ for these small values of $n$ due to integer-related constraints in the extremal configurations.

We can adjust the proof of Lemma~\ref{lem:3lb} to give a lower bound on the size of the largest monochromatic circuit in a $3$-coloring of the edges of $K_n$, as follows. First, as before, we can remove a forest in each color and leave each color class Eulerian. The resulting graph $G$ has at least $\binom{n}{2}-3n$ edges, so all but at most $3\sqrt{n}$ of the vertices have degree at least $n-1-2\sqrt{n}$. We then pass to the induced subgraph $G'$ on these $n'\geq n-3\sqrt{n}$ vertices; the minimum degree of $G'$ is at least $n'-1-2\sqrt{n}\geq n'-3\sqrt{n'}$ when $n$ is sufficiently large. The argument then proceeds largely as in the proof of Lemma~\ref{lem:3lb}, except that the condition of every green or blue component intersecting both $V_1$ and $V_2$ is strengthened by requiring every green or blue component to intersect each of $V_1$ and $V_2$ in more than $6\sqrt{n'}$ vertices. The proof then concludes as before, reducing to the case where there are exactly two components in each color. When the edges between $V(G')$ and $V(G)\setminus V(G')$ are added back in, it remains true that there are at most two components in each color with a positive number of edges. So, at least one monochromatic component in $G$ has at least $\frac{1}{6}(|E(G)|-(3\sqrt{n})^2) \geq \frac{1}{12}n^2 - O(n)$ edges. Since this component of $G$ is Eulerian, we have a circuit, and hence a trail, of the desired length, for all sufficiently large $n$. The upper bound from Lemma~\ref{lem:3ub} likewise applies to the size of the longest circuit, showing that the lower bound is asymptotically tight.

\vspace{3mm}

\noindent {\bf Acknowledgments.} I would like to thank my advisor Jacob Fox for introducing me to this problem and for helpful conversations along the way, as well as David Conlon and Mykhaylo Tyomkyn for the insights from our later joint work that served as inspiration for some of the strengthened arguments in the revised version of this paper. In addition, I would like to thank the anonymous referees for their careful reading and helpful comments, including a specific suggestion that led to a significant strengthening in the bound for general $k$ in Theorem~\ref{thm:colorlb}.

\bibliographystyle{acm}
\bibliography{main}

\end{document}